\documentclass[11pt,a4paper]{amsart}
\usepackage{etex} 
\usepackage[latin1]{inputenc}
\usepackage[T1]{fontenc}
\usepackage[right=3cm,left=3cm,top=3cm,bottom=4cm]{geometry}
\usepackage{ifthen}
\usepackage{enumerate}
\usepackage{amsmath}
\usepackage{amsthm}
\usepackage{amsfonts}
\usepackage{amssymb}
\usepackage{latexsym}
\usepackage{ae}

\title[Corrigendum]{Corrigendum to the paper 'On the ideal theorem for number fields` [Funct. Approximatio, Comment. Math. 53, No. 1, 31--45 (2015)]}
\author{Olivier Bordellès}
\address{2 allée de la combe \\ 43000 Aiguilhe \\ France}
\email{borde43@wanadoo.fr}
\date{}
\dedicatory{}

\newcommand{\R}{\mathbb {R}}
\newcommand{\C}{\mathbb {C}}

\theoremstyle{theorem}
\newtheorem{theorem}{Theorem}[section]
\newtheorem{prop}[theorem]{Proposition}

\theoremstyle{definition}

\theoremstyle{remark}
\newtheorem{rem}[theorem]{Remark}

\makeatletter
\@namedef{subjclassname@2020}{%
  \textup{2020} Mathematics Subject Classification}
\makeatother

\begin{document}

\begin{abstract}
This paper is a corrigendum to the article 'On the ideal theorem for number fields`. The main result of this paper proves to be untrue and is replaced by an estimate of a weighted sum with an improved error term.
\end{abstract}


\maketitle

\section{Introduction}
\label{s1}

\noindent
Let $K$ be an algebraic number field of degree $n$, $\zeta_K$ be the attached Dedekind zeta-function and $r_K$ be the ideal-counting function of $K$. It is customary to set
$$\Delta_K(x) := \sum_{m \leqslant x} r_K(m) - \kappa_K x$$
where $\kappa_K$ is the residue of $\zeta_K(s)$ at $s=1$. Using contour integration, Landau \cite[Satz~210]{lan} proved that
$$\Delta_K(x) \ll x^{1-\frac{2}{n+1} + \varepsilon}$$
which was slightly improved by Nowak \cite{now} for $4 \leqslant n \leqslant 9$ and by Lao \cite{lao} for $n \geqslant 10$ who showed that
$$\Delta_K(x) \ll x^{1-\frac{3}{n+6} + \varepsilon}.$$
In \cite{bor15}, an alternative method was used, starting with the Voronoï's identity, to show that, for $n \geqslant 4$
\begin{equation}
   \Delta_K(x) \ll x^{1-\frac{4}{2n+1} + \varepsilon}. \label{eq:bor}
\end{equation}
This result is based among other things upon the following Proposition \cite[Proposition~3.5]{bor15} stating that, if $X \geqslant 1$ is a real number, $1 \leqslant M < M_1 \leqslant 2M$ and $1 \leqslant N < N_1 \leqslant 2N$ are integers, $\alpha ,\beta \in \R$ are such that $(\alpha - 1) (\alpha - 2) \alpha \beta \neq 0$, and if $(a_m),(b_n) \in \C$ are two sequences of complex numbers such that $|a_m| \leqslant 1$ and $|b_n| \leqslant 1$, then, if $M \gg X$
\begin{multline*}
   \left( MN \right)^{-\varepsilon} \sum_{M < m \leqslant M_1} a_m \sum_{N < n \leqslant N_1} b_n e \left( X \left( \frac{m}{M} \right)^\alpha \left( \frac{n}{N} \right)^\beta \right) \\\ll  \left( XM^5N^7 \right)^{1/8} + N\left( X^{-2}M^{11} \right)^{1/12} + \left( X^{-3} M^{21} N^{23} \right)^{1/24} + M^{3/4}N  + X^{-1/4} M N
\end{multline*}
where, as usual, $e(x) = e^{2 i \pi x}$. Unfortunately, Roger Baker \cite{bak20} pointed to me the fact that this bound is impossible at least when $N=1$, since taking $a_m := e \left( -Xm^\alpha M^{-\alpha} \right)$ and $b_1=1$ implies that the sum on the left is $\asymp M$. An inspection of the proof reveals that \cite[Theorem~3]{rob} is improperly used, so that the very last part of the proof cannot be handled by this result. This entails that the estimate \eqref{eq:bor} remains still unproven, and the question of an improvement of Landau-Nowak's results in the cases $4 \leqslant n \leqslant 9$ is still open.

\section{A weighted sum}

\noindent
Our aim in this section is to show that adding a very small weight to the sum enables us to derive a better error term. More precisely, we will show the following estimate.

\begin{theorem}
\label{th:weight_sum}
Assume $n \geqslant 6$. Then
$$\sum_{m \leqslant x/e} r_K(m) \log \log \tfrac{x}{m} = E_1(1) \kappa_K x + O_{K,\varepsilon} \left( x^{1 - \frac{3}{n} + \varepsilon} \right)$$
where, for any $X>0$
$$E_1(X) := \int_X^\infty \frac{e^{-t}}{t} \, \mathrm{d}t$$
and hence $E_1(1) \approx 0.22$.
\end{theorem}

\subsection{Tools}

\begin{prop}
\label{t1-1}
Let $K$ be an algebraic number field of degree $n \geqslant 2$ and $x \geqslant d_K^{1/2}$ be a large real number. For any $\varepsilon \in \left( 0,\frac{1}{2}\right)$, we have
$$\frac{1}{x} \int_1^x \left( \sum_{m \leqslant t} r_K(m) - \kappa_K t \right) \mathrm{d}t = O_{K,\varepsilon} \left( x^{\lambda_n+\varepsilon} \right)$$
where
$$\lambda_n := \begin{cases} \frac{3}{4} - \frac{3}{2n}, & \text{if} \quad 2 \leqslant n \leqslant 6 \\ & \\ 1 - \frac{3}{n}, & \text{if} \quad n \geqslant 6. \end{cases}$$
Similarly
$$\sum_{m \leqslant t} r_K(m) \log \tfrac{x}{m} = \kappa_K x + O_{K,\varepsilon} \left( x^{\lambda_n+\varepsilon} \right).$$
\end{prop}

\begin{proof}
Note that it is equivalent to show that
$$\frac{1}{x} \int_1^x \left( \sum_{m \leqslant t} r_K(m) \right) \mathrm{d}t = \tfrac{1}{2}\kappa_K x + O_{K,\varepsilon} \left( x^{\lambda_n+\varepsilon} \right).$$
From Perron's formula
$$\frac{1}{x} \int_1^x \left( \sum_{m \leqslant t} r_K(m) \right) \textrm{d}t = \frac{1}{2 \pi i} \int_{2-i \infty}^{2+i \infty} \frac{\zeta_K(s)}{s(s+1)} \, x^s \, \textrm{d}s.$$
Assume first that $n \geqslant 6$. Moving the integration contour to the line $\sigma = 1 - \frac{3}{n} + \varepsilon$ and taking the residue of the function $\zeta_K(s) x^s s^{-1} (s+1)^{-1}$ into account, we get
\begin{align*}
   \frac{1}{x} \int_1^x \left( \sum_{m \leqslant t} r_K(m) \right) \textrm{d}t &= \tfrac{1}{2}\kappa_K x + \frac{1}{2 \pi i} \int_{1-3/n+\varepsilon - i \infty}^{1-3/n+\varepsilon+i \infty} \frac{\zeta_K(s)}{s(s+1)} \, x^s \, \textrm{d}s \\
   &:= \tfrac{1}{2}\kappa_K x + I(x).
\end{align*}
From \cite{hea}, we know that for any $t \in \R$
$$\zeta_K \left( \tfrac{1}{2} + it \right) \ll_{K,\varepsilon} \left( |t|+1 \right)^{n/6 + \varepsilon}$$
which implies using the Phragm\'{e}n-Lindel\"{o}f principle that
\begin{equation}
   \zeta_K \left( \sigma + it \right) \ll_\varepsilon \left( |t|+1 \right)^{n(1-\sigma)/3 + \varepsilon} \quad \left( \tfrac{1}{2} \leqslant \sigma \leqslant 1, \ t \in \R \right) \label{e1}
\end{equation}
and hence
\begin{align*}
   \left | I(x) \right | & \ll_\varepsilon x^{1-3/n+\varepsilon} \left( 1 + \int_1^\infty \left | \zeta_K \left( 1 - \tfrac{3}{n} + \varepsilon + it \right) \right | t^{-2} \, \textrm{d}t \right) \\
   & \ll_{K,\varepsilon} x^{1-3/n+\varepsilon} \left( 1 + \int_1^\infty t^{-1 - \varepsilon (\frac{1}{3}n-1)} \, \textrm{d}t \right) \\
   & \ll_{K,\varepsilon} x^{1-3/n+\varepsilon} 
\end{align*}
as announced. The case $2 \leqslant n \leqslant 6$ is similar, except that we move the integration contour to the line $\sigma = \frac{3}{4} - \frac{3}{2n} + \varepsilon$ and we replace \eqref{e1} by
$$\zeta_K \left( \sigma + it \right) \ll_\varepsilon \left( |t|+1 \right)^{n(3-4\sigma)/6 + \varepsilon} \quad \left( 0 \leqslant \sigma \leqslant \tfrac{1}{2}, \ t \in \R \right).$$
For the $2$nd sum, the arguments are similar since
$$\sum_{m \leqslant x} r_K(m) \log \tfrac{x}{m} = \frac{1}{2 \pi i} \int_{2-i \infty}^{2+i \infty} \frac{\zeta_K(s)}{s^2} \, x^s \, \textrm{d}s.$$
The proof is complete.
\end{proof}

\begin{rem}
This proposition implies that there exists a real number $x_0 \geqslant d_K^{1/2}$ such that
$$\sum_{m \leqslant x_0} r_K(m) = \kappa_K x_0 + O_{K,\varepsilon}\left( x_0^{1-3/n+\varepsilon} \right) \quad \left( n \geqslant 6 \right).$$
\end{rem}

\subsection{Proof of Theorem~\ref{th:weight_sum}}

\noindent
On the one hand, using partial summation
$$\sum_{m \leqslant x/e} r_K(m) \log \log \tfrac{x}{m} = \int_1^{x/e} \frac{1}{t \log(x/t)} \left( \sum_{m \leqslant t} r_K(m) \right) \, \textrm{d}t$$
and on the other hand
\begin{align*}
   \sum_{m \leqslant x/e} r_K(m) &= \sum_{m \leqslant x/e} r_K(m) \log \frac{x}{m} - \int_1^{x/e} \frac{1}{t \log^2(x/t)} \left( \sum_{m \leqslant t} r_K(m) \log \frac{x}{m} \right) \, \textrm{d}t \\
   &= \sum_{m \leqslant x/e} r_K(m) \log \frac{x}{em} + \sum_{m \leqslant x/e} r_K(m) \\
   &  \qquad - \int_1^{x/e} \frac{1}{t \log^2(x/t)} \left( \sum_{m \leqslant t} r_K(m) \left( \log \frac{t}{m} + \log \frac{x}{t} \right) \right) \, \textrm{d}t \\
\end{align*}
so that
\begin{multline*}
   0 = \sum_{m \leqslant x/e} r_K(m) \log \frac{x}{em} - \int_1^{x/e} \frac{1}{t \log^2(x/t)} \left( \sum_{m \leqslant t} r_K(m) \log \frac{t}{m} \right) \, \textrm{d}t \\
   - \int_1^{x/e} \frac{1}{t \log(x/t)} \left( \sum_{m \leqslant t} r_K(m) \right) \, \textrm{d}t.
\end{multline*}
Hence
\begin{multline*}
   \sum_{m \leqslant x/e} r_K(m) \log \log \tfrac{x}{m} = \sum_{m \leqslant x/e} r_K(m) \log \frac{x}{em} \\
   - \int_1^{x/e} \frac{1}{t \log^2(x/t)} \left( \sum_{m \leqslant t} r_K(m) \log \frac{t}{m} \right) \, \textrm{d}t
\end{multline*}
and using Theorem~\ref{t1-1} we get
\begin{align*}
   \sum_{m \leqslant x/e} r_K(m) \log \log \tfrac{x}{m} &= \kappa_K e^{-1} x + O \left( x^{1-3/n+\varepsilon} \right) \\
   & \qquad - \int_1^{x/e} \frac{1}{t \log^2(x/t)} \left( \kappa_K t + O \left( t^{1-3/n+\varepsilon} \right)  \right) \, \textrm{d}t \\   
   &= \kappa_K \left( \frac{x}{e} - \int_1^{x/e} \frac{\textrm{d}t}{\log^2(x/t)} \right)  + O \left( x^{1-3/n+\varepsilon} \right) \\
   &= \kappa_K x \left( e^{-1} - \int_1^{\log x} \frac{e^{-u}}{u^2} \, \textrm{d}u \right)  + O \left( x^{1-3/n+\varepsilon} \right)
\end{align*}
and integrating by parts provides
\begin{align*}
   \int_1^{\log x} \frac{e^{-u}}{u^2} \, \textrm{d}u &= - \left. \frac{e^{-u}}{u} \right|_{1}^{\log x} - \int_1^{\log x} \frac{e^{-u}}{u} \, \textrm{d}u \\
   &= e^{-1} - E_1(1) + E_1(\log x) - \frac{1}{x\log x}
\end{align*}
and the bounds 
$$\frac{1}{x \log(ex)} < E_1(\log x) < \frac{1}{x \log x} \quad \left( x \geqslant 2 \right)$$
imply
$$\int_1^{\log x} \frac{e^{-u}}{u^2} \, \textrm{d}u = e^{-1} - E_1(1) + O \left( \frac{1}{x\log x} \right)$$
completing the proof.
\qed

\section*{Acknowledgements}

\noindent
The author sincerely thanks Prof.~Roger Baker for having pointed out this mistake in the original paper.

\end{document}